\documentclass[a4paper,12pt]{article}

\usepackage{relsize,exscale}
\usepackage{color}
\usepackage{tikz,epic,eso-pic}

\usepackage{hyperref}
\usepackage[arrow,matrix]{xy}
\usepackage{makeidx}
\usepackage{bm}
\usepackage{latexsym}
\usepackage{amscd}
\usepackage{amsmath}
\usepackage{amssymb}
\usepackage{hhline}

\usepackage{stmaryrd}

\usepackage{amsthm}
\usepackage{float}
\usepackage{psfrag}

\usetikzlibrary{intersections}
\colorlet{examplefill}{yellow!80!black}

\theoremstyle{plain}
\newtheorem{theorem}{Theorem}[section]
\newtheorem{corollary}[theorem]{Corollary}

\newtheorem{proposition}[theorem]{Proposition}

\theoremstyle{definition}
\newtheorem{definition}[theorem]{Definition}
\newtheorem{remark}[theorem]{Remark}
\newtheorem{example}[theorem]{Example}

\newcommand{\Z}{\mathbb{Z}}

\newcommand{\R}{\mathbb{R}}
\newcommand{\C}{\mathbb{C}}

\newcommand{\PP}{\mathbb{P}}

\newcommand{\F}{{\mathbb F}}
\newcommand{\ch}{\operatorname{ch}}
\newcommand{\bch}{\operatorname{bch}}
\newcommand{\uch}{\operatorname{uch}}
\newcommand{\Sep}{\operatorname{Sep}}
\newcommand{\wall}{\operatorname{Wall}}

\newcommand{\A}{{\mathcal A}}

\newcommand{\FF}{{\mathcal F}}
\newcommand{\LL}{{\mathcal L}}

\newcommand{\barH}{\overline{H}}
\newcommand{\barA}{\overline{\A}}
\newcommand{\barC}{\overline{C}}
\newcommand{\barD}{\overline{D}}
\newcommand{\barF}{\overline{F}}
\newcommand{\barX}{\overline{X}}
\newcommand{\barZ}{\overline{Z}}
\newcommand{\barnabla}{\overline{\nabla}}

\newcommand{\id}{\operatorname{id}}

\newcommand{\dense}{{\operatorname{\mathsf D}}_\infty}

\definecolor{deepblue}{cmyk}{0,0.83,1,0.70}
\definecolor{gray}{cmyk}{0,0,0,0.3}
\definecolor{rred}{cmyk}{0,1,1,0}
\definecolor{chairo}{cmyk}{0,0.83,1,0.70}
\definecolor{roypur}{cmyk}{0.75,0.90,0,0.1}
\definecolor{darkorc}{cmyk}{0.40,0.80,0.20,0}
\definecolor{oliv}{cmyk}{0.64,0.00,0.75,0.56}
\definecolor{azuro}{cmyk}{1,1,0,0.46}

\DeclareMathOperator{\image}{im}

\DeclareMathOperator{\codim}{codim}

\title{Vanishing results for the Aomoto complex of real 
hyperplane arrangements via minimality}
\author{Pauline Bailet\thanks{Department of Mathematics, 
University of Bremen, 28344 Bremen, GERMANY 
Email: bailet@uni-bremen.de} and 
Masahiko Yoshinaga\thanks{Department of Mathematics, 
Hokkaido University, 
North 10, West 8, Kita-ku, 
Sapporo 060-0810, 
JAPAN 
E-mail: yoshinaga@math.sci.hokudai.ac.jp}}
\date{\today}
\begin{document}
\maketitle

\begin{abstract}
We prove vanishing results of the cohomology groups of Aomoto complex 
over arbitrary coefficient ring for real hyperplane arrangements. 
The proof is using minimality of arrangements and 
descriptions of Aomoto complex in terms of chambers. 

Our methods also provide a new proof for the vanishing theorem of 
local system cohomology groups which was first proved by 
Cohen, Dimca and Orlik.

\end{abstract}

\tableofcontents

\section{Introduction}

Theory of hypergeometric integrals originated from Gauss has been 
generalized to higher dimensions, which has applications in 
various area of mathematics and physics (\cite{ao-ki, koh-cft, var}). 
In the above generalization, the notion of the local system 
cohomology groups on the complement of a hyperplane arrangement 
plays a crucial role. 

Let $\A=\{H_1, \dots, H_n\}$ be an arrangement of affine hyperplanes in 
$\C^\ell$, $M(\A)=\C^\ell\smallsetminus\bigcup_{H\in\A}H$ be its complement. 
We also fix a defining equation $\alpha_i$ of $H_i$. 
An arrangement $\A$ is called essential if normal vectors of hyperplanes 
generate $\C^\ell$. The first homology group $H_1(M(\A), \Z)$ is  a 
free abelian group generated by the meridians $\gamma_1, \dots, \gamma_n$ 
of hyperplanes. We denote their dual basis by $e_1, \dots, e_n\in 
H^1(M(\A), \Z)$. The element $e_i$ can be identified with 
$\frac{1}{2\pi\sqrt{-1}}d\log\alpha_i$ via the de Rham isomorphism. 

The isomorphism class of a rank one complex local system $\LL$ is 
determined by a homomorphism $\rho:H_1(M(\A), \Z)\longrightarrow
\C^\times$, which is also determined by an $n$-tuple 
$q=(q_1, \dots, q_n)\in(\C^\times)^n$, where 
$q_i=\rho(\gamma_i)$. 

For a generic parameter $(q_1, \dots, q_n)$, it is known that 
the following vanishing result holds. 
\begin{equation}
\label{eq:typicalvanishing}
\dim H^k(M(\A), \LL)=
\left\{
\begin{array}{ll}
0,& \mbox{ if $k\neq \ell$, }\\
&\\
|\chi(M(\A))|,& \mbox{ if $k=\ell$.}
\end{array}
\right.
\end{equation}
Several sufficient conditions for the vanishing (\ref{eq:typicalvanishing}) 
have been known (\cite{ao-ki, koh}). Among others, Cohen, Dimca and Orlik 
(\cite{cdo}) proved the following. 
\begin{theorem}
\label{thm:cdo}
(CDO-type vanishing theorem) Suppose that $q_X\neq 1$ for 
each dense edge $X$ contained in the hyperplane at infinity. 
Then the vanishing (\ref{eq:typicalvanishing}) holds. 
(See \S \ref{subsec:os} below for terminologies). 
\end{theorem}
The above result is stronger than many other vanishing results. 
Indeed for the case $\ell=2$, it was proved in \cite{y-mini} that 
the vanishing (\ref{eq:typicalvanishing}) with additional property holds 
if and only if the assumption of Theorem \ref{thm:cdo} holds. 

The local system cohomology group $H^k(M(\A), \LL)$ is computed by 
using twisted de Rham complex $(\Omega_{M(\A)}^\bullet, d+\omega\wedge)$, 
with $\omega=\sum\lambda_i d\log\alpha_i$, 
where $\lambda$ is a complex number such that 
$\exp(-2\pi\sqrt{-1}\lambda_i)=q_i$. 
The algebra of rational differential forms $\Omega_{M(\A)}^\bullet$ 
has a natural $\C$-subalgebra $A_{\C}^\bullet(\A)$ 
generated by $e_i=\frac{1}{2\pi\sqrt{-1}}d\log\alpha_i$. This 
subalgebra is known to be isomorphic to the cohomology ring 
$H^\bullet(M(\A), \C)$ of $M(\A)$ (\cite{bri-tress}) 
and having a combinatorial 
description the so-called Orlik-Solomon algebra \cite{O-S} 
(see \S \ref{subsec:os} below for details). 
The Orlik-Solomon algebra provides a subcomplex 
$(A_\C^\bullet(\A), \omega\wedge)$ of the twisted de Rham complex, 
which is called the Aomoto complex. 
There exists a natural morphism 
\begin{equation}
\label{eq:morphism}
(A_\C^\bullet(\A), \omega\wedge)\hookrightarrow 
(\Omega_{M(\A)}^\bullet, d+\omega\wedge)
\end{equation}
of complexes. The Aomoto complex 
$(A_\C^\bullet(\A), \omega\wedge)$ 
has a purely combinatorial description. 
Furthermore, it can be considered as 
a linearization of the twisted de Rham complex 
$(\Omega_{M(\A)}^\bullet, d+\omega\wedge)$. 
Indeed, there exists a Zariski open subset $U\subset(\C^\times)^n$ 
which contains $(1, 1, \dots, 1)\in(\C^\times)^n$ such that 
(\ref{eq:morphism}) is quasi-isomorphic for $q\in U$ (\cite{esv, stv, nty}). 
However, they are not isomorphic in general. 

Vanishing results for the cohomology of the Aomoto complex are 
also proved by Yuzvinsky. 

\begin{theorem}
\label{thm:yuz}
(\cite{yuz, yuz-bos}) 
Let $\omega=\sum_{i=1}^n 2\pi\sqrt{-1}\lambda_ie_i\in A_\C^1(\A)$. 
Suppose $\lambda_X\neq 0$ for all dense edge $X$ in $L(\A)$. 
Then we have 
\begin{equation}
\label{eq:yuz}
\dim H^k(A_\C^\bullet(\A), \omega\wedge)=
\left\{
\begin{array}{ll}
0,& \mbox{ if $k\neq \ell$, }\\
&\\
|\chi(M(\A))|,& \mbox{ if $k=\ell$.}
\end{array}
\right.
\end{equation}
\end{theorem}
We note that the assumptions in 
Theorem \ref{thm:cdo} 
and 
Theorem \ref{thm:yuz} 
are 
somewhat complementary. 
For the first one requires nonresonant condition 
along the hyperplane at infinity, on the other hand, 
Theorem \ref{thm:yuz} imposes nonresonant condition on all dense edges 
in the affine space. 

Recently, Papadima and Suciu proved that for a torsion local system, 
the dimension of the local system cohomology group is bounded by 
that of Aomoto complex with finite field coefficients. 

\begin{theorem}
\label{thm:ps-sp}
(\cite{PS}) 
Let $p\in\Z$ be a prime. 
Suppose $\omega=\sum_{i=1}^n\lambda_ie_i\in A_{\F_p}^1(\A)$ and 
$\LL$ is the local system determined by 
$q_i=\exp(\frac{2\pi\sqrt{-1}}{p}\lambda_i)$. 
Then 
\begin{equation}
\label{eq:ineqps}
\dim_\C H^k(M(\A), \LL)\leq
\dim_{\F_p} H^k(A_{\F_p}^\bullet(\A), \omega \wedge), 
\end{equation}
for all $k\geq 0$. 
\end{theorem}
In view of Papadima and Suciu's inequality (\ref{eq:ineqps}), 
it is natural to expect that CDO-type vanishing theorem for 
a $p$-torsion local system may be deduced from that of the Aomoto complex 
with finite field coefficients. The main result of this paper is 
the following CDO-type vanishing theorem for Aomoto complex with 
arbitrary coefficient ring. 

\begin{theorem}
\label{thm:intromain}
(Theorem \ref{thmprincipal}) 
Let $\A=\{H_1, \dots, H_n\}$ be an essential affine hyperplane arrangement 
in $\R^\ell$. Let $R$ be a commutative ring with $1$. 
Let 
$\omega=\sum_{i=1}^n\lambda_ie_i\in A_R^1(\A)$. 
Suppose that $\lambda_X\in R^\times$ for any dense edge $X$ 
contained in the hyperplane at infinity. Then the following holds. 
\begin{equation}
\label{eq:main}
H^k(A_R^\bullet(\A), \omega\wedge)\simeq
\left\{
\begin{array}{ll}
0,& \mbox{ if $k\neq \ell$, }\\
&\\
R^{|\chi(M(\A))|},& \mbox{ if $k=\ell$.}
\end{array}
\right.
\end{equation}
\end{theorem}
Our proof relies on several works (\cite{y-lef, y-mini, y-cham}) 
concerning minimality of arrangements. 
We can also provide an alternative proof of 
Theorem \ref{thm:cdo} for real arrangements. 

This paper is organized as follows. 

In \S \ref{sec:notation}, 
we recall basic terminologies and the description of 
Aomoto complex in terms of chambers developed in 
\cite{y-lef, y-mini, y-cham}. We also recall 
the description of twisted minimal complex in terms 
of chambers. Simply speaking, two cochain complexes 
$(R[\ch^\bullet(\A)], \nabla_{\omega_\lambda})$ and 
$(\C[\ch^\bullet(\A)], \nabla_{\LL})$ are constructed 
by using the real structures of $\A$ (adjacent relations of chambers). 
These cochain complexes provide a parallel description 
between the cohomology of Aomoto complex and the local system 
cohomology group. 
Indeed, using these complexes, we can 
prove simultaneously CDO-type vanishing result 
for both cases. 

In \S \ref{sec:results}, 
we state the main result and describe the strategy for the proof. 
The proof consists of an easy part and a hard part. 
The easy part of the proof is done mainly by 
elementary arguments on cochain complex, which is also done 
in this section. The hard part is done in the subsequent section 
(\S \ref{sec:proofs}). 

The final section \S \ref{sec:proofs} is devoted to 
analyze the polyhedral structures of chambers 
which are required for matrix presentations of the 
coboundary map of $(R[\ch^\bullet(\A)], \nabla_{\omega_\lambda})$.

\section{Notations and Preliminaries}

\label{sec:notation}

\subsection{Orlik-Solomon algebra and Aomoto complex}

\label{subsec:os}

Let $\A=\{H_1,\hdots,H_n\}$ be an affine 
hyperplane arrangement in $V=\R^\ell$. 
Denote by $M(\A)=\C^\ell \smallsetminus \cup_{i=1}^n H_i \otimes \C$ 
the complement of the complexified hyperplanes.
By identifying $\R^\ell$ with $\PP_{\R}^\ell\smallsetminus\barH_\infty$, 
define the projective closure by 
$\barA=\{\barH_1,\hdots,\barH_n,\barH_\infty\}$, where 
$\barH_i\subset\PP_{\R}^\ell$ is the closure of $H_i$ in the projective space. 
We denote $L(\A)$ and $L(\barA)$ the intersection posets of 
$\A$ and $\barA$, respectively, namely, the poset of subspaces 
obtained as intersections of some hyperplanes with reverse inclusion order. 
An element of $L(\A)$ (and $L(\barA)$) is also called an edge. 
We denote by $L_k(\A)$ the set of all $k$-dimensional edges. 
For example $L_{\ell}(\A)=\{V\}$ and $L_{\ell-1}(\A)=\A$. 
Then $\A$ is essential if and only if $L_0(\A)\neq\emptyset$. 

Let $R$  be a commutative ring. Orlik and Solomon gave a 
simple combinatorial description of the algebra $H^*(M(\mathcal{A}),R)$, 
which is the quotient of the exterior algebra on 
classes dual to the meridians, modulo a certain 
ideal determined by $L(\A),$ see \cite{O-S}. 
More precisely, by associating to any hyperplane $H_i$ a generator $e_i \simeq \frac{1}{2\pi\sqrt{-1}} d \log \alpha_i,$ the Orlik-Solomon algebra $A^\bullet_R(\A)$ of $\A$ is the quotient of the exterior algebra generated by the elements $e_i,\,1\leq i \leq n,$ modulo the ideal $I(\A)$ generated by:
\begin{itemize}
\item the elements of the form $\{e_{i_1}\wedge\cdots\wedge  e_{i_s}\,|\,H_{i_1} \cap \cdots \cap H_{i_s}= \emptyset\},$
\item the elements of the form $\{\partial(e_{i_1}\wedge\cdots\wedge  e_{i_s})\,|\,H_{i_1} \cap \cdots \cap H_{i_s}\neq \emptyset\,\,\text{and}\,\,\codim(H_{i_1} \cap \cdots \cap H_{i_s})<s\}$, where 
$\partial(e_{i_1}\wedge\cdots\wedge  e_{i_s})=\sum_{\alpha=1}^s(-1)^{\alpha-1}
e_{i_1}\wedge\dots\wedge\widehat{e_{i_\alpha}}\wedge\dots\wedge e_{i_s}$. 
\end{itemize}
Let $\lambda=(\lambda_1, \dots, \lambda_n)\in R^n$ and $\omega_{\lambda}= {\sum_{i=1}^n \lambda_i e_i}\in A^1_R(\A).$ The cochain complex 
$(A^\bullet_R(\A),\omega_\lambda\wedge)=\{A^\bullet_R(\A)\stackrel{\omega_\lambda\wedge}{\longrightarrow}A^{\bullet+1}_R(\A)\}$ is called 
\emph{the Aomoto complex}. 


We say that an edge $X \in L(\barA)$ is \textit{dense} if the localization $\barA_X=\{\barH\in \barA \,|\, X \subseteq \barH \}$ is indecomposable (see 
\cite{ot-int} for more details). 
We consider each hyperplane $\barH\in\barA$ is a dense edge.
In this paper, the set of dense edges of $\barA$ contained in $\barH_\infty$ 
plays an important role. We denote by $\dense(\barA)$ the set of all 
dense edges contained in $\barH_\infty$. We will characterize 
$X\in\dense(\barA)$ in terms of chambers in Proposition \ref{densechamber}. 

Set $\lambda_\infty:= - \sum_{i=1}^n \lambda_i$, and for any $X\in L(\barA)$, 
$\lambda_X:= {\sum_{\barH_i\supset X} \lambda_i}$, where the index $i$ 
runs $\{1, 2, \dots, n, \infty\}$. 

The isomorphism class of 
a rank one local system $\LL$ on the complexified complement $M(\A)$ is 
determined by the monodromy $q_i\in\C^\times$ around each hyperplane 
$H_i$. As in the case of 
Aomoto complex, we denote $q_\infty=(q_1q_2\cdots q_n)^{-1}$ 
and $q_X=\prod_{\barH_i\supset X}q_i$ for an edge $X\in L(\barA)$.

\subsection{Chambers and minimal complex}

\label{subsec:chambers}

In this section, we recall the description of the minimal complex 
in terms of real structures from \cite{y-lef, y-mini, y-cham}. 
Let $\A=\{H_1, \dots, H_n\}$ be an essential 
hyperplane arrangement in $\R^\ell$. 
A connected component of $\R^\ell\smallsetminus\bigcup_{i=1}^nH_i$ is called a 
chamber. The set of all chambers of $\A$ is denoted by $\ch(\A)$. 
A chamber $C\in\ch(\A)$ is called a bounded chamber if $C$ is 
bounded. The set of all bounded chambers of $\A$ is denoted by 
$\bch(\A)$. 
For a chamber $C\in\ch(\A)$, 
denote by $\barC$ the closure of $C$ in $\PP_{\R}^\ell$.
It is easily seen that 
a chamber $C$ is bounded if and only if $\barC\cap\barH_\infty=\emptyset$. 

For given two chambers $C, C'\in\ch(\A)$, denote by 
\[
\Sep(C, C'):=\{H_i\in\A\mid \mbox{ $H_i$ separates $C$ and $C'$}\}, 
\]
the set of separating hyperplanes of $C$ and $C'$.

For the description of the minimal complex, we have to fix a 
generic flag. Let 
\[
\FF: \emptyset=F^{-1}\subset F^0\subset F^1\subset\dots\subset F^{\ell}
=\R^\ell
\]
be a generic flag (i.e., $F^k$ is a generic $k$-dimensional affine subspace, 
in other words, $\dim(\barX\cap\barF^k)=\dim \barX+k-\ell$ for any 
$\barX\in L(\barA)$). 
The genericity of $\FF$ is equivalent to 
\[
F^k\cap L_i(\A)=L_{k+i-\ell}(\A\cap F^k), 
\]
for $k+i\geq\ell$.

\begin{definition}
\label{def:nearinfty}
We say that the hyperplane $F^{\ell-1}$ is near to $\barH_\infty$ when 
$F^{\ell-1}$ does not separate $0$-dimensional 
edges $L_0(\A)\subset\R^\ell$. 
Similarly, we say the flag $\FF$ is near to $\barH_\infty$ when 
$F^{k-1}$ does not separate $L_0(\A\cap F^{k})$ for all 
$k=1, \dots, \ell$. 
\end{definition}
From this point, we assume that the flag $\FF$ is 
near to $\barH_\infty$. 
For a generic flag $\FF$ near to $\barH_\infty$, we define 
\[
\begin{split}
\ch^k(\A)
&=
\{C\in\ch(\A)\mid C\cap F^k\neq\emptyset, C\cap F^{k-1}=\emptyset\}\\
\bch^k(\A)
&=
\{C\in\ch^k(\A)\mid C\cap F^k\mbox{ is bounded}\}\\
\uch^k(\A)
&=
\{C\in\ch^k(\A)\mid C\cap F^k\mbox{ is unbounded}\}.\\
\end{split}
\]
Then clearly, we have 
\[
\begin{split}
\ch^k(\A)&=\bch^k(\A)\sqcup\uch^k(\A)\\ 
\ch(\A)&=\bigsqcup_{k=0}^\ell\ch^k(\A).
\end{split}
\]
Note that $\bch^\ell(\A)=\bch(\A)$, however, for $k<\ell$, 
$C\in\bch^k(\A)$ is an unbounded chamber. 

\begin{definition}(\cite[Definition 2.1]{y-mini})
Let $C\in \bch(\A).$ 
There exists a unique chamber, denoted by $C^{\vee}\in\uch(\A)$, 
which is the opposite with respect to $\overline{C}\cap \barH_\infty,$ 
where $\overline{C}$ is the closure of $C$ in the projective space 
$\PP_{\R}^\ell$. 
\end{definition}

\begin{figure}[htbp]
\centering
\begin{tikzpicture}[scale=1]


\draw[thick,rounded corners=0.6cm] (1,1) -- (8,1) -- (8,4.5) node[right] {$\overline{H}_\infty$};

\draw[thick] (1,2.5) node [left] {$H_1$} --(8.5,3.5);
\draw[thick] (1,3.5) node [left] {$H_2$} --(8.5,1.5);

\draw[thick,rounded corners=0.3cm] (5,0.5) -- (3,1.5) -- (3,4.5) node [above] {$H_3$}; 
\draw[thick,rounded corners=0.3cm] (4,0) -- (4,4.5) node [above] {$H_4$}; 
\draw[thick,rounded corners=0.3cm] (3,0.5) -- (5,1.5) -- (5,4.5) node [above] {$H_5$}; 

\filldraw[fill=black, draw=black] (4,1) circle (2pt);

\draw (2,1.5) node[above] {$C_1$};
\draw (3.5,1.5) node[above] {$C_2$};
\draw (4.5,1.5) node[above] {$C_3$};
\draw (6,1.5) node[above] {$C_4$};

\draw (2,3.5) node[above] {$C_4^\lor$};
\draw (3.5,3.5) node[above] {$C_2^\lor$};
\draw (4.5,3.5) node[above] {$C_3^\lor$};
\draw (6,3.5) node[above] {$C_1^\lor$};

\draw (2,0) node[above] {$C_1^\lor$};
\draw (3.5,0) node[above] {$C_3^\lor$};
\draw (4.5,0) node[above] {$C_2^\lor$};
\draw (6,0) node[above] {$C_4^\lor$};

\draw[very thick] (1,1) node[above] {\footnotesize $\overline{C_1}\cap\overline{H}_\infty$} --(4,1);

\end{tikzpicture}
\caption{Opposite chambers}
\label{fig:opposite}
\end{figure}
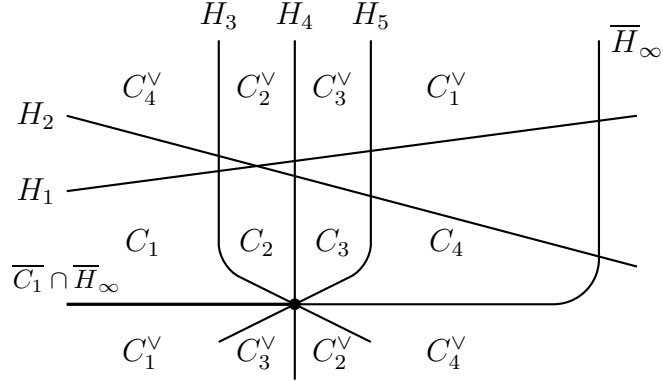

Let us denote the projective subspace generated by $\barC\cap\barH_\infty$ 
by $X(C)=\langle\barC\cap\barH_\infty\rangle$. 

\begin{proposition}
\label{prop:charsep}Let $C\in \bch(\A),$ then
\begin{equation}
\label{eq:charsep}
\Sep(C, C^\lor)=\{H\in\A\mid \barH\not\supset X(C)\}=
\barA\smallsetminus\barA_{X(C)}. 
\end{equation}
\end{proposition}

\begin{proof}
Let $p\in C$ and $p'$ be a point in the relative interior of 
$\barC\cap\barH_\infty$. 
Take the line $L=\langle p, p'\rangle\subset\PP_{\R}^\ell$. 
Choose a point $p''\in C^\lor\cap L$. 
Then consider the segment 
$[p, p'']\subset \R^\ell=\PP_{\R}^\ell\smallsetminus\barH_\infty$ 
(See Figure \ref{fig:segment}). 
On the projective space $\PP_{\R}^\ell$, the line $L=\langle p, p'\rangle$ 
must intersect every hyperplane $\barH\in\barA$ exactly once. 
Furthermore, $L$ intersects $\barH\in\barA_{X(C)}$ at $p'$. 
On the other hand, the segment $[p, p'']$ intersects 
$H\in\Sep(C, C^\lor)$.  
Hence we have (\ref{eq:charsep}). 
\end{proof}

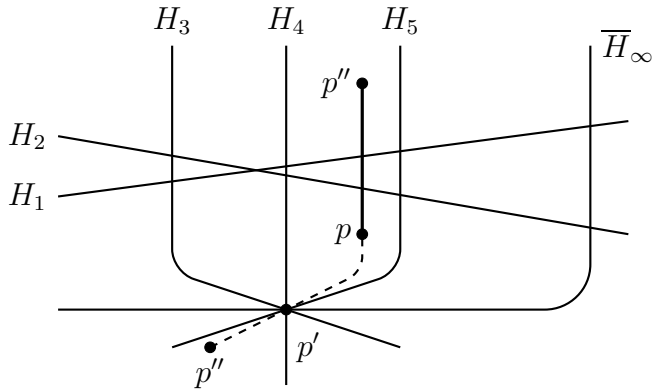
\begin{figure}[htbp]
\centering
\begin{tikzpicture}[scale=1]


\draw[thick,rounded corners=0.6cm] (1,1) -- (8,1) -- (8,4.5) node[right] {$\overline{H}_\infty$};

\draw[thick] (1,2.5) node [left] {$H_1$} --(8.5,3.5);
\draw[thick] (1,3.3) node [left] {$H_2$} --(8.5,2);

\draw[thick,rounded corners=0.3cm] (5.5,0.5) -- (2.5,1.5) -- (2.5,4.5) node [above] {$H_3$}; 
\draw[thick,rounded corners=0.3cm] (4,0) -- (4,4.5) node [above] {$H_4$}; 
\draw[thick,rounded corners=0.3cm] (2.5,0.5) -- (5.5,1.5) -- (5.5,4.5) node [above] {$H_5$}; 

\filldraw[fill=black, draw=black] (3,0.5) node [below] {$p''$} circle (2pt);
\filldraw[fill=black, draw=black] (5,4) node [left] {$p''$} circle (2pt);
\filldraw[fill=black, draw=black] (5,2) node [left] {$p$} circle (2pt);
\filldraw[fill=black, draw=black] (4,0.5) node [right] {$p'$};

\draw[thick,dashed,rounded corners=0.2cm] (3,0.5)--(5,1.5)--(5,2);
\draw[very thick] (5,2)--(5,4);

\filldraw[fill=black, draw=black] (4,1) circle (2pt);

\end{tikzpicture}
\caption{The segment $[p, p'']$ (thick segment).}
\label{fig:segment}
\end{figure}

\begin{corollary}
\label{cor:sep}
If $\dim X(C)=\ell-1$, then $\Sep(C, C^\lor)=\A$. 
\end{corollary}

\begin{proof}
In this case, $\barA_{X(C)}=\{\barH_\infty\}$. Proposition 
\ref{prop:charsep} concludes $\Sep(C, C^\lor)=\A$. 
\end{proof}

\begin{proposition}
\label{prop:bchuch}
(\cite{y-lef, y-mini})
\begin{itemize}
\item[$(i)$] $\#\ch^k(\A)=b_k$, where $b_k=b_k(M(\A))$. 
\item[$(ii)$] $\#\bch^k(\A)=\#\uch^{k+1}(\A)$. 
\item[$(iii)$] $\#\bch^k(\A)=b_k-b_{k-1}+\dots+(-1)^kb_0$. 
\end{itemize}
\end{proposition}
Concerning $(ii)$ of Proposition \ref{prop:bchuch}, 
an explicit bijection is given by the opposite chamber, 
\[
\iota: \bch^k(\A)\stackrel{\simeq}{\longrightarrow}\uch^{k+1}(\A), 
C\longmapsto C^\lor. 
\]

Next result characterizes the dense edge contained in $\barH_\infty$. 

\begin{proposition}(\cite[Proposition 2.4]{y-mini})\label{densechamber}
Let $\A$ be an affine arrangement in $\R^\ell.$ 
An edge $X\in L(\barA)$ with $X \subseteq \barH_\infty$ 
is dense if and only if $X=X(C)$ for some chamber 
$C\in \uch(\A).$ In particular, we have 
\begin{equation}
\label{eq:denseedge}
\dense(\A)=\{X(C)\mid C\in\uch(\A)\}. 
\end{equation}
\end{proposition}

Next we define the degree map 
\[
\deg:\ch^k(\A)\times\ch^{k+1}(\A)\longrightarrow\Z. 
\]
Let $B=B^k\subset F^k$ be a 
$k$-dimensional ball with sufficiently large radius so that 
every $0$-dimensional edge $X\in L_0(\A\cap F^k)\simeq L_{\ell-k}(\A)$ 
is contained in the interior of $B^k$. 
Let $C\in\ch^k(\A)$ and $C'\in\ch^{k+1}(\A)$. 
Then there exists a vector field $U^{C'}$ on $F^k$ 
(\cite{y-lef}) which satisfies the following conditions. 
\begin{itemize}
\item 
$U^{C'}(x)\neq 0$ for $x\in\partial \barC\cap B^k$. 
\item 
Let $x\in\partial(B^k)\cap \barC$. Then $T_x(\partial B^k)$ can be considered 
as a hyperplane of $T_xF^k$. We impose a condition that 
$U^{C'}(x)\in T_xF^k$ is contained in the half space corresponding to 
the inside of $B^k$. 
\item 
If $x\in H\cap F^k$ for a hyperplane $H\in\A$, then 
$U^{C'}(x)\not\in T_x(H\cap F^k)$ and is directed to the side 
in which $C'$ is lying with respect to $H$. 
\end{itemize}
When the vector field $U^{C'}$ satisfies the above conditions, 
we say that \emph{the vector field $U^{C'}$ is directed to the chamber $C'$}. 
The above conditions imply that if either 
$x\in H\cap F^k$ or $x\in \partial B^k$, 
then $U^{C'}(x)\neq 0$. 
Thus for $C\in\ch^k(\A)$, $U$ is not vanishing on $\partial (\barC\cap B^k)$. 
Hence we can consider the following Gauss map. 
\[
\frac{U^{C'}}{|U^{C'}|}: \partial(\barC\cap B^k)\longrightarrow S^{k-1}. 
\]
Fix an orientation of $F^k$,  which induces an orientation on 
$\partial(\barC\cap B^k)$. 

\begin{definition}
\label{degreemap}
Define the degree $\deg(C, C')$ between $C\in\ch^k(\A)$ and 
$C'\in\ch^{k+1}(\A)$ by 
\[
\deg(C,C'):=\deg\left(\left.
\frac{U^{C'}}{|U^{C'}|}\right|_{\partial (\barC\cap B^k)}:
\partial(\barC\cap B^k)\longrightarrow S^{k-1}\right)\in\Z. 
\]
This is independent of the choice of $U^{C'}$ (\cite{y-lef}). 
\end{definition}

If the vector field $U^{C'}$ does not have zeros on $\barC\cap B^k$, 
then the Gauss map can be extended to the map 
$\barC\cap B^k\longrightarrow S^{k-1}$. Hence 
$\frac{U^{C'}}{|U^{C'}|}:\partial(\barC\cap B^k)\longrightarrow S^{k-1}$ 
is homotopic to a constant map. Thus we have the following. 

\begin{proposition}
\label{prop:degree0}
If the vector field $U^{C'}$ is nowhere zero on $\barC\cap B^k$, 
then $\deg(C, C')=0$. 
\end{proposition}

\begin{example}
\label{ex:pointing}
Let $p_0\in F^k$ such that $p_0\notin\bigcup_{H\in\A}H\cup\partial B^k$. 
Define the pointing vector field $U^{p_0}$ by 
\begin{equation}
U^{p_0}(x)=\overrightarrow{x; p_0}\in T_xF^k, 
\end{equation}
where $\overrightarrow{x; p_0}$ is a tangent vector at $x$ pointing 
$p_0$ (see Figure \ref{fig:pointingvf}). 
The vector field $U^{p_0}$ is directed to the chamber which contains $p_0$. 
Note that $U^{p_0}(x)=0$ if and only if $x=p_0$. 
Hence if $p_0\notin C\cap B^k$, the Gauss map 
$\frac{U^{p_0}}{|U^{p_0}|}:\partial(\barC\cap B^k)\longrightarrow S^{k-1}$ 
has 
$\deg\left(\frac{U^{p_0}}{|U^{p_0}|}\right)=0$. Otherwise, if 
$p_0\in C\cap B^k$, $\deg\left(\frac{U^{p_0}}{|U^{p_0}|}\right)=(-1)^k$. 

\begin{figure}[htbp]
\centering
\begin{tikzpicture}[scale=1]


\draw [thick] (0,1)--(8,1); 
\draw [thick] (0,4)--(8,4); 
\draw [thick] (1,0)--(1,5); 
\draw [thick] (7,0)--(7,5); 
\draw [thick] (0,0.5)--(8,4.5);


\draw[->, thick] (1,4)--(2.25,3.5);
\draw[->, thick] (1,3)--(2.25,2.75);
\draw[->, thick] (1,2)--(2.25,2);
\draw[->, thick] (1,1)--(2.25,1.25);

\draw[->, thick] (2,4)--(3,3.5);
\draw[->, thick] (2,3)--(3,2.75);
\draw[->, thick] (2,1)--(3,1.25);

\draw[->, thick] (3,4)--(3.75,3.5);
\draw[->, thick] (3,3)--(3.75,2.75);
\draw[->, thick] (3,2)--(3.75,2);
\draw[->, thick] (3,1)--(3.75,1.25);

\draw[->, thick] (4,4)--(4.5,3.5);
\draw[->, thick] (4,3)--(4.5,2.75);
\draw[->, thick] (4,2)--(4.5,2);
\draw[->, thick] (4,1)--(4.5,1.25);

\draw[->, thick] (5,4)--(5.25,3.5);
\draw[->, thick] (5,3)--(5.25,2.75);
\draw[->, thick] (5,2)--(5.25,2);
\draw[->, thick] (5,1)--(5.25,1.25);

\draw[->, thick] (6,4)--(6,3.5);
\draw[->, thick] (6,3)--(6,2.75);
\draw[->, thick] (6,1)--(6,1.25);

\draw[->, thick] (7,4)--(6.75,3.5);
\draw[->, thick] (7,3)--(6.75,2.75);
\draw[->, thick] (7,2)--(6.75,2);
\draw[->, thick] (7,1)--(6.75,1.25);

\filldraw[fill=black, draw=black] (6,2) circle (2pt) node[left] {$p_0$};

\end{tikzpicture}
\caption{Pointing Vector field $\frac{1}{4}U^{p_0}$}
\label{fig:pointingvf}
\end{figure}
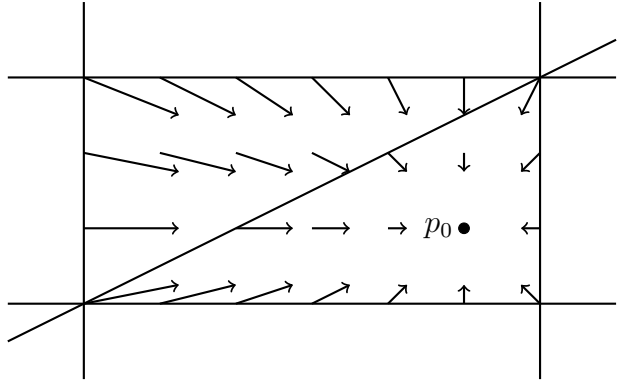
\end{example}

Consider the Orlik-Solomon algebra $A_R^\bullet(\A)$ over 
the commutative ring $R$. Let $\omega_\lambda=\sum_{i=1}^n\lambda_i e_i
\in A_R^1(\A)$, ($\lambda_i\in R$). 
We will describe the Aomoto complex $(A_R^\bullet(\A), \omega_\lambda\wedge)$ 
in terms of chambers. 
For two chambers $C, C'\in\ch(\A)$, define $\lambda_{\Sep(C, C')}$ by 
\[
\lambda_{\Sep(C, C')}:=
\sum_{H_i\in\Sep(C, C')}\lambda_i. 
\]
\begin{proposition}
\label{prop:lambdadense}
Let $C$ be an unbounded chambers. Then 
\[
\lambda_{\Sep(C, C^\lor)}=-\lambda_{X(C)}. 
\]
\end{proposition}
\begin{proof}
By Proposition \ref{prop:charsep}, we have 
$\barA=\barA_{X(C)}\sqcup\Sep(C, C^\lor)$. Hence, 
from the definition of $\lambda_{\infty}=-\sum_{i=1}^n\lambda_i$, 
we obtain 
$\lambda_{\Sep(C, C^\lor)}+\lambda_{X(C)}=0$. 
\end{proof}

Let $R[\ch^k(\A)]=\bigoplus_{C\in\ch^k(\A)}R\cdot[C]$ 
be the free $R$-module generated by $\ch^k(\A)$. 
Let $\nabla_{\omega_\lambda}: R[\ch^k(\A)] \longrightarrow R[\ch^{k+1}(\A)]$ 
be the $R$-homomorphism defined by 
\begin{equation}
\label{eq:defwedge}
\nabla_{\omega_\lambda}([C])= 
\displaystyle{\sum_{C'\in \ch^{k+1}}} \deg(C,C')\cdot 
\lambda_{\Sep(C, C')}\cdot [C'].
\end{equation}

\begin{proposition}
\label{prop:aomotochamber}
(\cite{y-cham})
$(R[\ch^\bullet(\A)],\nabla_{\omega_\lambda})$ is a cochain complex. Furthermore, there is a natural isomorphism of cochain complexes,
\[
(R[\ch^\bullet(\A)],\nabla_{\omega_\lambda}) 
\simeq (A^\bullet_R(\A),\omega_\lambda\wedge). 
\]
In particular, 
\[
H^k
(R[\ch^\bullet(\A)],\nabla_{\omega_\lambda}) 
\simeq 
H^k(A^\bullet_R(\A),\omega_\lambda\wedge). 
\]
\end{proposition}
Let $\LL$ be a rank one local system on $M(\A)$ which 
has monodromy $q_i\in\C^\times$ ($i=1, \dots, n$) around $H_i$. 
Fix $q_i^{1/2}=\sqrt{q_i}$ and define $q_{\infty}^{1/2}$ and 
$\Delta(C, C')$ by 
$q_{\infty}^{1/2}:=\left(q_1^{1/2}\cdots q_n^{1/2}\right)^{-1}$ and 
\[
\Delta(C, C'):=
\prod_{H_i\in\Sep(C, C')}q_i^{1/2}-
\prod_{H_i\in\Sep(C, C')}q_i^{-1/2}, 
\]
respectively. Then the local system cohomology group 
can be computed in a similar way to the Aomoto complex. 
Indeed, let us define the linear map 
$\nabla_{\LL}:\C[\ch^k(\A)]\longrightarrow\C[\ch^{k+1}(\A)]$ by 
\[
\nabla_{\LL}([C])= 
\sum_{C'\in \ch^{k+1}} \deg(C,C')\cdot 
\Delta(C, C')\cdot [C'].
\]
Then we have the following. 

\begin{proposition}
\label{prop:localchamber}
(\cite{y-lef})
$(\C[\ch^\bullet(\A)],\nabla_{\LL})$ is a cochain complex. 
Furthermore, there is a natural isomorphism of cohomology groups: 
\[
H^k(\C[\ch^\bullet(\A)],\nabla_{\LL}) \simeq H^k(M(\A), \LL).
\]
\end{proposition}


\section{Main results and strategy} 

\label{sec:results}

\subsection{Main theorems}

In this section, let $\A=\{H_1, \dots, H_n\}$ be a hyperplane 
arrangement in $\R^\ell$ and $R$ be a commutative ring with $1$.

\begin{theorem}
\label{thmprincipal}
If $\lambda_X \in R^\times$ for all 
$X\in\dense(\barA)$,  
then
\[
H^k(\C[\ch^\bullet(\A)],\nabla_{\omega_\lambda})\simeq
\left\{
\begin{array}{ll}
0,& \mbox{ if }k<\ell, \\
&\\
R[\bch(\A)],& \mbox{ if }k=\ell. 
\end{array}
\right.
\]
\end{theorem}
More generally, we can prove the following. 

\begin{corollary}
\label{cor:principal}
Let $0\leq p<\ell$. 
If $\lambda_X \in R^\times$ for all 
$X\in \dense(\barA)$ with $\dim(X)\geq p,$ then
\[
H^k(\C[\ch^\bullet(\A)],\nabla_{\omega_\lambda})=0, \mbox{ for all }
0\leq k < \ell-p.
\]
\end{corollary}
\proof
Here we give a proof of Corollary \ref{cor:principal} based on 
the main Theorem \ref{thmprincipal}. 
If we consider $\A\cap F^{\ell-p}$. The Orlik-Solomon algebra 
$A_R^{\bullet}(\A\cap F^{\ell-p})$ is isomorphic to 
$A_R^{\leq \ell-p}(\A)$. Hence we have an isomorphism 
\begin{equation}
\label{eq:propagate}
H^k(A_R^{\bullet}(\A\cap F^{\ell-p}), \omega_\lambda\wedge)\simeq
H^k(A_R^{\bullet}(\A), \omega_\lambda\wedge), 
\end{equation}
for $k<\ell-p$. Note that $L(\A\cap F^{\ell-p})\simeq L^{\geq p}(\A)$. 
By the assumption, we have $\lambda_X\in R^\times$ for 
any $X\in\dense(\A\cap F^{\ell-q})$. Hence by Theorem \ref{thmprincipal}, 
the left hand side of (\ref{eq:propagate}) is vanishing. 
\endproof

By Proposition \ref{prop:aomotochamber}, we have the following 
vanishing theorem for the Aomoto complex. 
\begin{corollary}
Let $0\leq p<\ell$. 
If $\lambda_X \in R^\times$ for all 
$X\in \dense(\barA)$ with $\dim(X)\geq p,$ then
\[
H^k(A_R^\bullet(\A),\omega_\lambda\wedge)=0, \mbox{ for all }
0\leq k < \ell-p.
\]
\end{corollary}

\begin{remark}
\label{rem:locsyscdo}
Completely similar proof works also for the case of local systems. 
Namely, if the local system $\LL$ satisfies that $q_X \neq 1$ for all 
$X\in \dense(\barA)$ with $\dim(X)\geq p,$ then 
\[
H^k(\C[\ch^\bullet(\A)],\nabla_{\LL})=0, \mbox{ for all }k<\ell-p. 
\]
Using Proposition \ref{prop:localchamber}, this implies 
\[
H^k(M(\A), \LL)=0, \mbox{ for all }k<\ell-p, 
\]
which gives an alternative proof for Theorem \ref{thm:cdo} 
by Cohen, Dimca and Orlik. 
\end{remark}

\subsection{Strategy for the proof of Theorem \ref{thmprincipal}}

\label{subsec:strategy}

In order to analyze the cohomology group,
\[
H^k(R[\ch^\bullet(\A)], \nabla_\omega)=
\frac
{\ker\left(\nabla_\omega:R[\ch^k(\A)]\longrightarrow R[\ch^{k+1}(\A)]\right)}
{\image\left(\nabla_\omega:R[\ch^{k-1}(\A)]\longrightarrow R[\ch^{k}(\A)]\right)}, 
\]
we will use the direct decomposition 
$R[\ch^k(\A)]=R[\bch^k(\A)]\oplus R[\uch^k(\A)]$, and then consider 
the map 
\begin{equation}
\overline{\nabla}_{\omega_{\lambda}}: 
R[\bch^k(\A)] \hookrightarrow R[\ch^{k}(\A)] 
\stackrel{\nabla_\omega}{\longrightarrow} 
R[\ch^{k+1}(\A)]\twoheadrightarrow R[\uch^{k+1}(\A)]. 
\end{equation}
We will study the map 
$\overline{\nabla}_{\omega_{\lambda}}: 
R[\bch^k(\A)] \longrightarrow R[\uch^{k+1}(\A)]$ in detail below. 
Recall that there is a natural bijection $\iota: 
\bch^k(\A)\stackrel{\simeq}{\longrightarrow}\uch^{k+1}(\A)$ 
(see Proposition \ref{prop:bchuch} and subsequent remarks), 
once we fix an ordering $C_1, \dots, C_b$ of $\bch^k(\A)$, 
we obtain a matrix expression of the map 
$\overline{\nabla}_{\omega_{\lambda}}$.  
We will prove the following. 

\begin{itemize}
\item[(i)] 
Let $C\in\bch^k(\A)$. Then 
$\deg(C, C^\lor)=(-1)^{\ell-1-\dim X(C)}$. 
\item[(ii)] 
For an appropriate ordering of $\bch^k(\A)=\{C_1, \dots, C_b\}$, 
the matrix expression of 
$\overline{\nabla}_{\omega_{\lambda}}:R[\bch^{k}(\A)]\longrightarrow
R[\uch^{k+1}(\A)]$ becomes 
an upper-triangular matrix. 
\item[(iii)] 
$\det\barnabla_\omega\in R^\times$
\item[(iv)] 
These imply Theorem \ref{thmprincipal}. 
\end{itemize}
(i) and (ii) will be proved in \S \ref{sec:proofs}. 

Here we prove (iii) and (iv) based on (i) and (ii). 
First note that from Proposition \ref{prop:lambdadense}, the 
definition (\ref{eq:defwedge}) of the coboundary map of the 
complex $(R[\ch^\bullet(\A)], \nabla_\omega)$, 
and uppertriangularity (ii) above, we have 
\[
\det\barnabla_\omega=\pm\prod_{C\in\bch^k(\A)}\deg(C, C^\lor)\lambda_{X(C)}. 
\]
From the assumption that $\lambda_{X}\in R^\times$ for 
$X\in\dense(\A)$ (see also Proposition \ref{densechamber}), 
we obtain (iii). 
Since 
$\barnabla_\omega:R[\bch^k(\A)]\stackrel{\simeq}{\longrightarrow} 
R[\uch^{k+1}(\A)]$ 
is an isomorphism of free $R$-modules, which are diagonals of 
the following diagram, 
we have $H^k(R[\ch^\bullet(\A)], \nabla_{\omega})=0$ for $k<\ell$ and 
$H^\ell(R[\ch^\bullet(\A)], \nabla_{\omega})\simeq R[\bch^\ell(\A)]$. 
\[
\begin{array}{ccccccccccc}
R[\ch^0]&\stackrel{\nabla_\omega}{\longrightarrow}&R[\ch^1]&\stackrel{\nabla_\omega}{\longrightarrow}&\cdots&\stackrel{\nabla_\omega}{\longrightarrow}&R[\ch^k]&\stackrel{\nabla_\omega}{\longrightarrow}&R[\ch^{k+1}]&\stackrel{\nabla_\omega}{\longrightarrow}&\cdots\\
||&&||&&&&||&&||&&\\
R[\bch^0]&&R[\bch^1]&&\cdots&&R[\bch^k]&&R[\bch^{k+1}]&&\\
&\searrow&\oplus&\searrow&&\searrow&\oplus&\searrow&\oplus&&\\
&&R[\uch^1]&&\cdots&&R[\uch^k]&&R[\uch^{k+1}]&&
\end{array}
\]

\section{Proofs}

\label{sec:proofs}

In this section, we prove (i) and (ii) in \S \ref{subsec:strategy} for 
$k=\ell-1$. Namely: 
\begin{itemize}
\item[(i')] For a chamber $C\in\bch^{\ell-1}(\A)$, 
$\deg(C, C^\lor)=(-1)^{\ell-1-\dim X(C)}$. 
\item[(ii')] For an appropriate ordering of $\{C_1, \dots, C_b\}=
\bch^{\ell-1}(\A)$, the matrix expression of $\overline{\nabla}_{\omega_{\lambda}}:R[\bch^{\ell-1}(\A)]\longrightarrow R[\uch^\ell(\A)]$ becomes 
an upper-triangular matrix. 
\end{itemize}
For other $k<\ell$, the assertions are proved by a similar way using 
the generic section by $F^{k+1}$ (see the argument of the proof of 
Corollary \ref{cor:principal}). 

\subsection{Structure of Walls}

For simplicity we will set $F=F^{\ell-1}$. Recall that 
$\bch^{\ell-1}(\A)=\{C\in\ch(\A)\mid C\cap F\mbox{ is a bounded 
chamber of $F\cap\A$}\}$. 
Let $C\in\bch^{\ell-1}(\A)$. A hyperplane $H\in\A$ is said to be 
a wall of $C$ if $H\cap F$ is a supporting hyperplane of 
a facet of $\barC\cap F$. 
For any $C\in \bch^{\ell-1}(\A)$, we denote by 
$\wall(C)$ the set of all walls of $C$. 

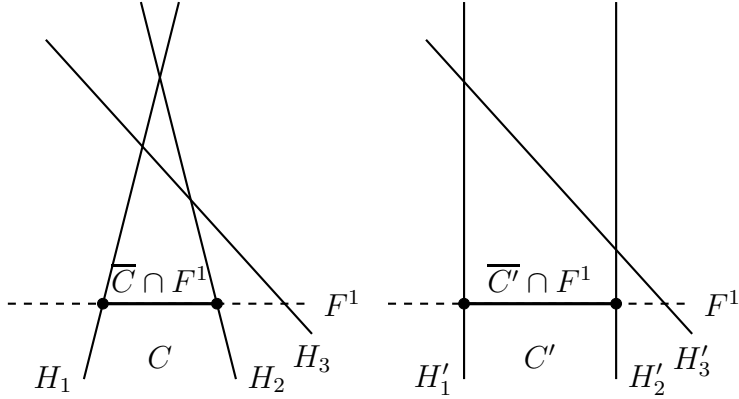
\begin{figure}[htbp]
\centering
\begin{tikzpicture}[scale=1]


\draw[dashed,thick] (0,1)--(4,1) node[right] {$F^1$};
\draw[thick] (1,0) node[left] {$H_1$} -- (2.25,5);
\draw[thick] (3,0) node[right] {$H_2$}--(1.75,5);
\draw[thick] (4,0.6) node[below] {$H_3$} --(0.5,4.5); 
\filldraw[fill=black, draw=black] (1.25,1) circle (2pt);
\filldraw[fill=black, draw=black] (2.75,1) circle (2pt);
\draw[very thick] (1.25,1)--node[above] {$\overline{C}\cap F^1$} (2.75,1); 
\draw (2,0) node [above] {$C$}; 

\draw[dashed,thick] (5,1)--(9,1) node[right] {$F^1$};
\draw[thick] (6,0) node[left] {$H'_1$} -- (6,5);
\draw[thick] (8,0) node[right] {$H'_2$}--(8,5);
\draw[thick] (9,0.6) node[below] {$H'_3$} --(5.5,4.5); 
\filldraw[fill=black, draw=black] (6,1) circle (2pt);
\filldraw[fill=black, draw=black] (8,1) circle (2pt);
\draw[very thick] (6,1)--node[above] {$\overline{C'}\cap F^1$} (8,1); 
\draw (7,0) node [above] {$C'$};


\end{tikzpicture}
\caption{\small $\wall(C)=\wall_2(C)=\{H_1, H_2\}, \wall(C')=\wall_1(C')=\{H'_1, H'_2\}$}
\label{fig:walls}
\end{figure}

We divide the set of walls into two types. 

\begin{definition}
\label{def:walls}
A wall $H\in\wall(C)$ is called the first kind if 
$\barH\supset X(C)$. Otherwise $H$ is called a wall of second kind. 
The set of walls of first kind, and second kind are denoted by 
$\wall_1(C)$ and $\wall_2(C)$ respectively. We have 
$\wall(C)=\wall_1(C)\sqcup\wall_2(C)$. 
(See Figure \ref{fig:walls} and \ref{fig:walls12}.) 
\end{definition}

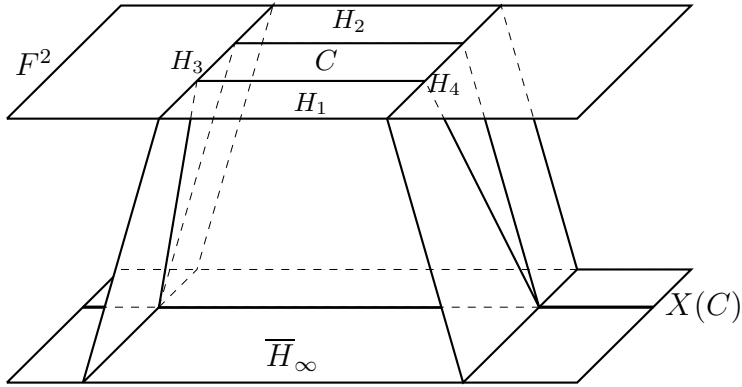
\begin{figure}[htbp]
\centering
\begin{tikzpicture}[scale=1]


\draw[thick] (0,0)--(1.5,1.5)--(9,1.5)--(7.5,0)--node[above]{$\overline{H}_{\infty}$}(0,0);

\draw[very thick] (1,1)--(8.5,1) node[right] {$X(C)$}; 

\filldraw[fill=white, draw=black, thick] (1,0)--(2.5,1.5)--(3.5,5)--(2,3.5)--(1,0);

\filldraw[fill=white, draw=black, thick] (3,4.5)--(6,4.5)--(7,1)--(2,1)--(3,4.5);

\filldraw[fill=white, draw=black, thick] (2.5,4)--(5.5,4)--(7,1)--(2,1)--(2.5,4);

\filldraw[fill=white, draw=black, thick] (5,3.5)--(6,0)--(7.5,1.5)--(6.5,5)--(5,3.5);

\draw[thick] (5.5,4)--(7,1)--(6,4.5);

\filldraw[fill=white, draw=black, thick] (0,3.5)--node [left] {$F^2$}(1.5,5)--(9,5)--(7.5,3.5)-- (0,3.5);

\draw[thick] (2,3.5)-- node [left] {\footnotesize $H_3$} (3.5,5); 
\draw[thick] (5,3.5)-- node [below] {\footnotesize $H_4$} (6.5,5); 
\draw[thick] (2.5,4)-- node [below] {\footnotesize $H_1$} (5.5,4); 
\draw[thick] (3,4.5)-- node [above] {\footnotesize $H_2$} (6,4.5); 
\draw[thick] (4.2,4) node[above] {\small $C$}; 

\draw[dashed, very thin] (1.5,1.5)--(9,1.5);

\draw[dashed, very thin] (1,1)--(8.5,1) ;

\draw[dashed, very thin] (2,1)--(2.5,1.5)--(3.5,5);

\draw[dashed, very thin] (2,1)--(3,4.5)--(6,4.5)--(7,1);

\draw[dashed, very thin] (2,1)--(2.5,4)--(5.5,4)--(7,1);

\draw[dashed, very thin] (7.5,1.5)--(6.5,5);

\end{tikzpicture}
\caption{$\wall_1(C)=\{H_1, H_2\}, \wall_2(C)=\{H_3, H_4\}$.}
\label{fig:walls12}
\end{figure}

Let $C\in \bch^{\ell-1}(\A)$ and $\wall_1(C)=\{H_{i_1}, \dots, H_{i_k}\}$ the 
walls of first kind. We choose defining equations 
$\alpha_{i_1}, \dots, \alpha_{i_k}$ of $\wall_1(C)$ so that 
\[
C\subset\{\alpha_{i_1}>0\}\cap\dots\cap\{\alpha_{i_k}>0\}. 
\]
Note that $\widetilde{C}:=
\{\alpha_{i_1}>0\}\cap\dots\cap\{\alpha_{i_k}>0\}$ is a chamber of 
$\wall_1(\A)$. 
Let $D\in\uch(\A)$ be another unbounded chamber of $\A$. 
Then $D$ is said to be inside $\wall_1(C)$ if 
\[
D\subset\widetilde{C}=
\{\alpha_{i_1}>0\}\cap\dots\cap\{\alpha_{i_k}>0\}. 
\]
This condition is also equivalent to 
$\Sep(C,D)\cap \wall_1(C)=\emptyset$. 

Recall that the opposite chamber of $C\in\bch^{\ell-1}(\A)$ is defined 
as the opposite chamber with respect to $X(C)\subset\barH_\infty$. 
Using (\ref{eq:charsep}), we have the following. 

\begin{proposition}
\label{prop:cap}
Let $C\in\bch^{\ell-1}(\A)$. Then $\Sep(C, C^\lor)\cap\wall(C)=\wall_2(C)$. 
\end{proposition}


\begin{remark}\label{rkinside}
Let $C\in\bch^{\ell-1}(\A)$. 
If $D$ is inside the walls of $\wall_1(C),$ then we have 
$X(D)\subset X(C)$ and 
$\dim X(D)\leq \dim X(C).$
\end{remark}

\subsection{Fibered structure of chambers}

\label{subsec:fiber}

Let $d=\dim X(C)$. 
Let $C\in\bch^{\ell-1}(\A)$. 
As above, we let $\widetilde{C}\in
\ch(\wall_1(C))$ the unique chamber such that $C\subset\widetilde{C}$. 

For each point $p\in\overline{\widetilde{C}}$, 
denote by $G_1(p):=\langle X(C), p\rangle\cap F$ (Figure \ref{fig:basepoly}). 
Then $G_1(p)$ is a $d$-dimensional affine subspace which is 
parallel to each $H\in\wall_1(C)$. Fix a base point 
$p_0\in\widetilde{C}$. We also fix an $(\ell-1-d)$-dimensional subspace 
$G_2(p_0)\subset F$ 
which is passing through $p_0$ and 
transversal to $G_1(p_0)$ (see Figure \ref{fig:basepoly}). 
Let us call 
$Q_0:=G_2(p_0)\cap\overline{\widetilde{C}}$ the base polytope. 

Consider the map 
$\pi_C:\barC\cap F\longrightarrow Q_0, 
p\longmapsto G_1(p)\cap Q_0$. 
For each $q\in Q_0$, the fiber $\pi_C^{-1}(q)=G_1(q)\cap \barC$ 
is a $d$-dimensional polytope. 
This fact is a conclusion of the assumption that $F$ is 
generic and near to $\barH_\infty$, and the following 
elementary proposition. 

\begin{figure}[htbp]
\centering
\begin{tikzpicture}[scale=1]


\draw[thick] (0,0)--(8,0)--(9,4)--(1,4)node[left]{$F$}--(0,0);

\draw[thick] (0.25,1)node[left]{$H_1$}--(8.25,1);
\draw[thick] (0.75,3)node[left]{$H_2$}--(8.75,3);

\draw[thick] (0.5,0.5)--(2.5,3.5);
\draw[thick] (2.5,0.5)--(5.5,3.5);
\draw[thick] (4,3.5)--(5,0.5);
\draw (3,2.3)node[above]{$\overline{C}\cap F$};

\draw[very thick, blue] (0.5,2)--(8.5,2);
\draw[blue] (5.5,2) node[above]{$G_1(p_0)$}; 

\draw[thick,red!70!black] (6.5,0)--(7.5,4);
\draw[red!70!black] (6.6,0.5) node[right]{$G_2(p_0)$};

\filldraw[fill=green!50!black, draw=green!50!black] (6.75,1) circle (2pt);
\filldraw[fill=green!50!black, draw=green!50!black] (7.25,3) circle (2pt);
\draw[green!50!black, very thick] (6.75,1)--(7.25,3) ;
\draw[green!50!black] (7.6,3) node[below] {$Q_0$};

\filldraw[fill=black, draw=black] (7,2) circle (2pt) ;
\draw(7.3,2) node[below] {$p_0$};

\draw[thick, red] (0.3,1.2)--(8.3,1.2) node[right] {$G_1(p)$};
\filldraw[fill=red, draw=red] (2,1.2)  circle (2pt) node[above,red] {$p$}; 
\filldraw[fill=red, draw=red] (6.8,1.2)  circle (2pt) ;
\draw (6.8,1.5) node[left, red] {$\pi_C(p)$};

\end{tikzpicture}
\caption{Base polytope $Q_0$ ($\wall_1(C)=\{H_1, H_2\}$)}
\label{fig:basepoly}
\end{figure}
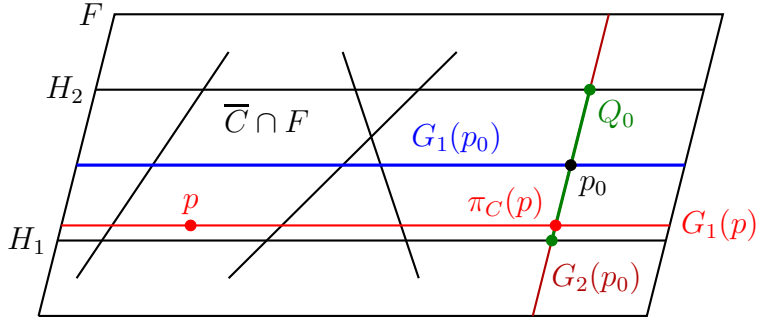

\begin{proposition}
Let $P\subset\R^\ell$ be an $\ell$-dimensional polytope. Let $X\subset P$ be 
a $d$-dimensional face ($0\leq d\leq \ell$). We denote by $\langle X\rangle$ 
the $d$-dimensional affine subspace spanned by $X$. Then for $\varepsilon
\in\R^\ell$ with sufficiently small 
$0\leq |\varepsilon|\ll 1$, 
$(\langle X\rangle+\varepsilon)\cap P$ is either an empty set or 
a $d$-dimensional polytope. 
\end{proposition}

\begin{remark}
\label{rem:section}
Since $\pi_C:\barC\cap F\longrightarrow Q_0$ is a fibration 
with contractible fibers, there exists a continuous 
section $\sigma_C:Q_0\longrightarrow\barC\cap F$ such that 
$\pi_C\circ\sigma_C=\id_{Q_0}$. 
\end{remark}

\subsection{Upper-triangularity}

\label{subsec:uppertri}

Let us fix an ordering of chambers of $\bch^{\ell-1}(\A)=\{C_1,\hdots,C_b\}$ 
in such a way that 
\[
\dim X(C_1)\geq \dim X(C_2) \geq \cdots \geq \dim X(C_b).
\]
The main result in this section is the following. 

\begin{theorem}
\label{thm:triangular}
The matrix $(\deg(C_i, C_j^\lor))_{i, j=1, \dots, b}$ is 
upper-triangular. In other words, if $i>j$, 
$\deg(C_i, C_j^\lor)=0$. 
\end{theorem}
\begin{proof}

Let $C, D\in\bch^{\ell-1}(\A)$. Suppose 
$\dim X(D)\geq\dim X(C)$ and $C\neq D$. 
Then we will prove $\deg(C, D^\lor)=0$. 
The idea of 
the proof is to construct a vector field $U^{D^\lor}$ directed to $D^\lor$ 
on $F$ 
which is nowhere vanishing on a neighbourhood of 
$\barC\cap F\subset F$. 
Then by Proposition \ref{prop:degree0}, we have $\deg(C, D^\lor)=0$. 

We divide into three cases. 
\begin{itemize}
\item[(a)] 
$\dim X(C)=\ell-1$. 
\item[(b)] 
$\dim X(C)<\ell-1$ and $D$ is not inside of $\wall_1(C)$. 
\item[(c)] 
$\dim X(C)<\ell-1$ and $D$ is inside of $\wall_1(C)$. 
\end{itemize}

Firstly we consider the case (a). In this case, since $\dim X(D)\geq 
\dim X(C)$, we have $\dim X(D)=\ell-1$. 
Choose a point $p\in D\cap F$, and define the vector field 
$U$ on $F$ by 
\[
U(x)=\overrightarrow{x;p}\in T_xF. 
\]
Then the vector field is directed to $p$ and nowhere vanishing on 
$\barC\cap F$ (because $p\notin\barC$). 
By Corollary \ref{cor:sep}, $-U$ is a vector field directed to 
$D^\lor$, which is also nowhere vanishing on $\barC\cap F$. 
Hence $\deg(C, D^\lor)=0$. 

From now on, we assume $\dim X(C)<\ell-1$. If $D$ is inside of 
$\wall_1(C)$, then $X(D)\subset X(C)$ by Remark \ref{rkinside}, we 
have $\barA_{X(D)}\supset\barA_{X(C)}$. Proposition \ref{prop:cap} 
indicates $\Sep(D, D^\lor)\cap\barA_{X(C)}=\emptyset$. 
We can conclude that 
$D^\lor$ is also inside $\wall_1(C)$. Conversely, if $D$ is 
not inside of $\wall_1(C)$, then also $D^\lor$ is not inside $\wall_1(C)$. 

Next we consider the case (b). Then $\Sep(C, D^\lor)\cap\wall_1(C)
\neq\emptyset$. Choose a hyperplane $H_{i_0}\in\Sep(C, D^\lor)\cap\wall_1(C)$. 
Let $\alpha_{i_0}$ be the defining equation of $H_{i_0}$. Without loss of 
generality, we may assume that 
\[
\begin{split}
H_{i_0}^{+}&=\{\alpha_{i_0}>0\}\supset D^\lor\\
H_{i_0}^{-}&=\{\alpha_{i_0}<0\}\supset C. 
\end{split}
\]
We will construct a vector field $U^{D^\lor}$ on $F$ which 
is directed to $D^\lor$ and satisfying 
\begin{equation}
\label{eq:positivity}
U^{D^\lor}(x)\alpha_{i_0}>0, 
\end{equation}
for $x\in\barC\cap F$, where the left hand side of (\ref{eq:positivity}) is 
the derivative of $\alpha_{i_0}$ with respect to the vector field. 
In particular, we obtain a vector field directed to $D^\lor$ which 
is nowhere vanishing on $\barC\cap F$. It is enough to show that, at 
any point $x_0\in\barC$, there exists a local vector field around $x_0$ 
which satisfies (\ref{eq:positivity}). Then we will obtain a 
global vector field which satisfies (\ref{eq:positivity}) 
using partition of unity. 

It is sufficient to show the existence of such vector field 
around each vertex $x_0$ 
of $\barC\cap F$. By genericity of $F$, 
$Z:=\bigcap\A_{x_0}=
\bigcap_{x_0\in H\in\A}H$ is a $1$-dimensional flat of $\A$, 
which is transversal to $F$. By the assumption that $F$ does not separate 
$0$-dimensional flats of $\A$, we have 
\begin{equation}
\label{eq:inclusionZ}
\barZ\cap\barH_\infty
\subset\barC\cap\barH_\infty. 
\end{equation}
(See Figure \ref{fig:Z}.) 

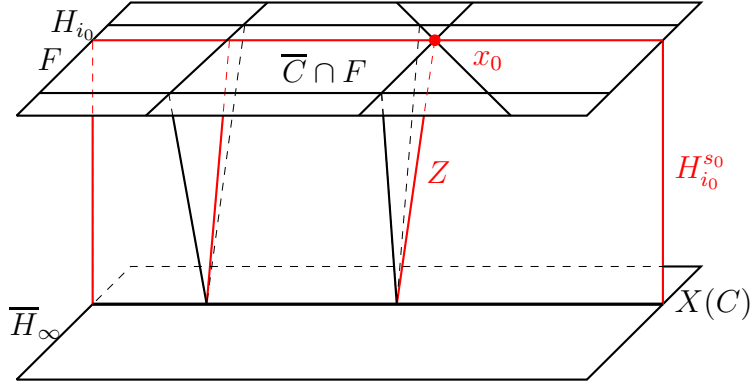
\begin{figure}[htbp]
\centering
\begin{tikzpicture}[scale=1]


\draw[thick] (0,0)--node[left]{$\overline{H}_{\infty}$}(1.5,1.5)--(9,1.5)--(7.5,0)--(0,0);


\filldraw[fill=white, draw=red, thick] (1,1)--(1,4.5)  -- 
(8.5,4.5) -- node [right, red] {$H_{i_0}^{s_0}$} (8.5,1) --(1,1);

\draw[red, thick] (2.8,4.5)--(2.5,1);

\draw[thick] (2,3.8) -- (2.5,1); 

\draw[thick] (4.8,3.8) -- (5,1); 

\draw[red, thick] (5.5,4.5) -- (5,1); 

\filldraw[fill=white, draw=black, thick] (0,3.5)--node [left] {$F$} 
(1.5,5)--(9,5)--(7.5,3.5)-- (0,3.5);

\draw[red, thick] (1,4.5)  -- (8.5,4.5); 
\draw[dashed, red, very thin] (1,1)--(1,4.5);
\draw[dashed, red, very thin] (2.8,4.5)--(2.5,1);

\draw[dashed, very thin] (0,0)--(1.5,1.5)--(9,1.5);

\draw[very thick] (1,1)--(8.5,1) node[right] {$X(C)$}; 

\draw[thick] (1.7,3.5)--(3.3,5);

\draw[thick] (1.2,4.7) node [left] {$H_{i_0}$} --(8.7,4.7);

\draw[thick] (0.3,3.8)-- node[above] {$\overline{C}\cap F$} (7.8,3.8);

\draw[thick] (4.5,3.5)--(6,5);

\draw[thick] (5,5)--(6.5,3.5);

\filldraw[fill=red, draw=red] (5.5,4.5) circle (2pt) ;
\draw (6.2,4.5) node[below, red]{$x_0$};

\draw[dashed, red, very thin] (5.5,4.5) -- node [right] {$Z$} (5,1); 

\draw[dashed, very thin] (4.8,3.8) -- (5,1);

\draw[dashed, very thin] (2,3.8) -- (2.5,1); 
\draw[dashed, very thin] (3,4.7) -- (2.5,1); 
\draw[dashed, very thin] (4.8,3.8) -- (5,1); 
\draw[dashed, very thin] (5.3,4.7) -- (5,1); 

\end{tikzpicture}
\caption{$Z$ and $H_{i_0}^{s_0}$.}
\label{fig:Z}
\end{figure}

Set $s_0:=\alpha_{i_0}(x_0)$ and $H_{i_0}^{s_0}=\{\alpha_{i_0}=s_0\}$ 
the hyperplane passing through $x_0$ which is parallel to $H_{i_0}$. 
Then we have $Z\subset H_{i_0}^{s_0}$, otherwise, contradicting 
(\ref{eq:inclusionZ}). The hyperplanes $\A_{x_0}=\A_Z$ determines 
chambers (cones), one of which, denoted by $\Gamma$, contains 
$D^\lor$ (Figure \ref{fig:gamma}). 
Hence the tangent vector $U^{D^\lor}(x_0)$ should be 
contained in $\Gamma$. Furthermore, 
\begin{equation}
D
\subset\Gamma\cap H_{i_0}^{+}
\subset\Gamma\cap H_{i_0}^{>s_0}. 
\end{equation}
In particular, we have 
$\Gamma\cap H_{i_0}^{>s_0}\neq\emptyset$. Thus we can construct 
a vector field $U^{D^\lor}$ around $x_0$ so that 
$U^{D^\lor}(x_0)\in \Gamma\cap H_{i_0}^{>s_0}$. 
Then (\ref{eq:positivity}) is satisfied around $x_0$. 
Hence we have $\deg(C, D^\lor)=0$ for the case (b). 

\begin{figure}[htbp]
\centering
\begin{tikzpicture}[scale=1]


\fill[green!20!white] (4,2) -- (4.707,1.293) arc (-45:45:1cm) --(4,2);
\draw[green!20!black] (4.8,1.5) node[right] {$\Gamma$};

\draw[thick] (0,0)--(8,0)--(9,4)--(1,4)node[left]{$F$}--(0,0);

\draw[thick] (0.25,1)--(8.25,1);
\draw[thick] (0.75,3)node[left]{$H_{i_0}$}--(8.75,3);

\draw[thick] (0.5,0.5)--(2,3.5);
\draw[thick] (2.5,0.5)--(5.5,3.5);
\draw[thick] (2.5,3.5)--(5.5,0.5);
\draw (2,1)node[above]{$\overline{C}\cap F$};

\draw[very thick, red] (0.5,2) node [left] {$H_{i_0}^{s_0}$} --(8.5,2);
\draw[->, thick, red] (1.75,2)--(2,2.5)node[right]{$H_{i_0}^{\geq s_0}$};

\draw[->,blue!50!black, very thick] (4,2) -- (5.5,2.5) node[right] {$U^{D^\lor}(x_0)$};

\filldraw[fill=black, draw=black] (4,2) circle (2pt) node[below]{$x_0$};

\end{tikzpicture}
\caption{Construction of the vector field $U^{D^\lor}$}
\label{fig:gamma}
\end{figure}
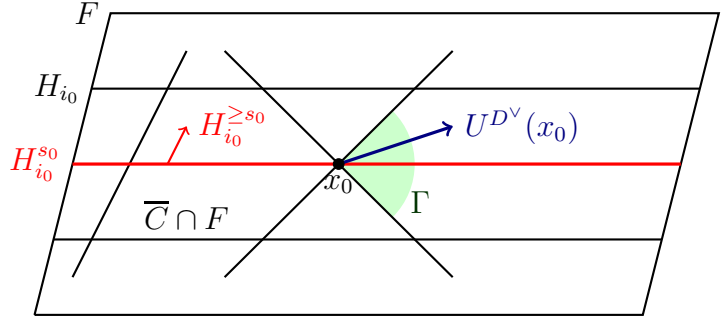

Thirdly, suppose $D$ is inside of $\wall_1(C)$, equivalently, 
$D\subset\widetilde{C}$. Let us handle the case (c). 
Since $X(D)\subset X(C)$ and $\dim X(D)\geq \dim X(C)$, we have 
$X(D)=X(C)$. In this case, $\wall_1(C)=\wall_1(D)$ and 
$\widetilde{C}=\widetilde{D}$. 
We consider the fibration 
$\pi_D:\barD\cap F\longmapsto Q_0$ 
which also has $d$-dimensional polytopes as fibers. 
Since the fiber is contractible, there exists a 
continuous section $\sigma_D:Q_0\longmapsto\barD\cap F$ 
such that $\pi_D\circ\sigma_D=\id_{Q_0}$. 

Now we construct a vector field. For each $p\in\barC\cap F$, 
we denote $G_2(p)$ the $(\ell-1-d)$-dimensional subspace 
which is passing through $p$ and parallel to $G_2(p_0)$ (Figure \ref{fig:V2}). 
Let $\{p'\}=G_2(p)\cap G_1(p_0)$. The tangent space is decomposed 
as $T_pF=T_pG_1(p)\oplus T_pG_2(p)$. We first construct a vector field 
on the second component. Let us define the tangent vector 
$V_2(p)\in T_pG_2(p)\subset T_pF$ by 
\begin{equation}
V_2(p)=\overrightarrow{p; p'}. 
\end{equation}
The vector field $V_2$ is obviously inward with respect to $\wall_1(C)$, 
and vanishing on the reference fiber $G_1(p_0)\cap \barC$. 

\begin{figure}[htbp]
\centering
\begin{tikzpicture}[scale=1.2]


\draw[thick] (0,0)--(8,0) node [right] {$F$} -- (8,4) -- 
(0,4)--cycle; 

\draw[thick] (0,1) -- (8,1); 
\draw[thick] (0,3) -- (8,3); 
\draw[thick] (0.5,0.5) -- (1.5,3.5); 
\draw[thick] (3.5,3.5)--(4.5,0.5);
\draw[thick] (3.5,0.5)--(5.5,3.5);
\draw[thick] (5.7,0.5)--(6.2,3.5);

\draw (2,1) node[above] {$C$};

\filldraw[fill=blue, draw=blue] (6.8,3) circle (2pt) ;
\filldraw[fill=blue, draw=blue] (7.2,1) circle (2pt) ;
\draw[thick, blue] (6.8,3)--(7,2)--node[right]{\small $Q_0$} (7.2,1);

\filldraw[fill=black, draw=black] (7,2) circle (2pt) ;
\draw (7.2,2) node [above, above, black] {\small $p_0$} ;

\draw[red, thick] (0,2)--(8,2) node[right] {$G_1(p_0)$}; 

\draw[thick, blue] (2.7, 3.5) node[above] {\small $G_2(p)$} --(3.3,0.5); 
\filldraw[fill=blue, draw=blue] (2.9,2.5) node [right] {$p$} circle (2pt) ;

\filldraw[fill=blue, draw=blue] (3,2) circle (2pt) ;
\draw (2.8,2) node[below, blue] {$p'$}; 

\foreach \x in {0.5,1,1.5,2,2.5,3.5,4}
\draw[->, blue, thick] (\x, 3)--(\x+0.1,2.5);

\foreach \x in {0.5,1,1.5,2.5,3.5,4}
\draw[->, blue, thick] (\x, 1)--(\x-0.1,1.5);

\foreach \x in {0.3,0.8,1.3,1.8,2.3,3.8}
\draw[->, blue, thick] (\x, 2.5)--(\x+0.05,2.25);

\draw[->, blue, thick] (2.9,2.5)--(2.95,2.25);

\foreach \x in {0.7,1.2,1.7,2.2,3.7}
\draw[->, blue, thick] (\x, 1.5)--(\x-0.05,1.75);

\end{tikzpicture}
\caption{$V_2$.}
\label{fig:V2}
\end{figure}

Next we construct a vector field $V_1$ along the fibers $G_1(p)$. 
Using the section $\sigma_C:Q_0\longrightarrow\barC\cap F$ 
(Remark \ref{rem:section}), 
define $V_1$ by 
\begin{equation}
V_1(p)=\overrightarrow{p;\sigma_D(\pi_C(p))}, 
\end{equation}
(Figure \ref{fig:V1}). 

\begin{figure}[htbp]
\centering
\begin{tikzpicture}[scale=1.2]


\draw[thick] (0,0)--(8,0) node [right] {$F$} -- (8,4) -- 
(0,4)--cycle; 

\draw[thick] (0,1) -- (8,1); 
\draw[thick] (0,3) -- (8,3); 
\draw[thick] (0.5,0.5) -- (1.5,3.5); 
\draw[thick] (2.5,3.5)--(3.5,0.5);
\draw[thick] (2.5,0.5)--(4.5,3.5);
\draw[thick] (6.2,0.5)--(6.7,3.5);

\draw (2,1) node[above] {$C$};
\draw (5,1) node[above] {$D$};

\draw[rounded corners=0.3cm, red, thick] (4,1)--(4,1.5)--
(4.5,2)--(5,2.5)node[right] {\small $\sigma_D(Q_0)$}--(4.8,3);

\filldraw[fill=red, draw=red] (6.8,3) circle (2pt) ;
\filldraw[fill=red, draw=red] (7.2,1) circle (2pt) ;
\draw[thick, red] (6.8,3)--(7,2)--node[right]{\small $Q_0$} (7.2,1);


\filldraw[fill=red, draw=red] (1,2) node[left, red] {\small $p$} circle (2pt); 
\draw[thin, red] (1,2)--(4.5,2); 
\filldraw[fill=red, draw=red] (4.5,2) node[right, red] {\small $p''$} circle (2pt); 

\foreach \y in {1.05,1.5,2,2.5,3.05}
\draw[->, red, thick] (1,\y)--(1.7,\y);

\foreach \y in {1.05,1.5,2,2.5,3.05}
\draw[->, red, thick] (2.5,\y)--(3.2,\y);

\end{tikzpicture}
\caption{$V_1, p''=\sigma_D(\pi_C(p))$.}
\label{fig:V1}
\end{figure}
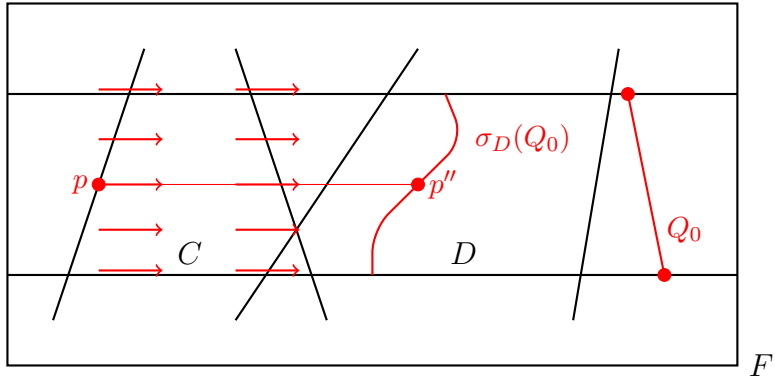

\begin{proposition}
For sufficiently large $t\gg 0$, the vector field $tV_1+V_2$ is 
directed to $D$. Similarly, $-tV_1+V_2$ is a vector field directed to 
$D^\lor$. 
\end{proposition}
\begin{proof}
Let $p\in H\in\wall_1(C)$. Recall that $D$ is inside $\wall_1(C)$. 
Since $V_2$ is inward and $V_1$ is tangent to $H$, the vector field 
$\pm tV_1+V_2$ is also 
inward. 
Let $H\in\wall_2(C)$ and $p\in H\cap F$. 
Then $V_1$ 
(resp. $-V_1$) is directed to $D$ (resp. $D^\lor$) with respect to $H$. 
Hence for sufficiently large $t$, $tV_1+V_2$ (resp. $-tV_1+V_2$) 
is directed to $D$ (resp. $D^\lor$). 
\end{proof}

Since $V_1$ is nowhere vanishing vector field on $\barC\cap F$, 
$-tV_1+V_2$ is a nowhere vanishing vector field around $\barC\cap F$ 
which is directed to $D^\lor$. 
Hence $\deg(C, D^\lor)=0$. 
This completes the proof of Theorem \ref{thm:triangular}. 
\end{proof}

\subsection{The degree formula}

\label{subsec:degree}

This section is devoted to prove the following. 

\begin{theorem}
\label{thm:degree}
Let $C\in\bch^{\ell-1}(\A)$. Suppose $\dim X(C)=d$. Then 
\begin{equation}
\deg(C, C^\lor)=(-1)^{\ell-1-d}. 
\end{equation}
\end{theorem}

We construct a vector field around $\barC\cap F$ which is 
directed to $C^\lor$. The vector field $V_2$ is the same as 
in the previous section (\S \ref{subsec:uppertri}). 
Define the vector field $V_1$ along the fibers $\pi_C$ by 
\begin{equation}
V_1(p)=\overrightarrow{p;\sigma_C(\pi_C(p))}
\end{equation}
(see Figure \ref{fig:degreeformula}). 

\begin{figure}[htbp]
\centering
\begin{tikzpicture}[scale=1.2]


\draw[thick] (0,0)--(8,0) node [right] {$F$} -- (8,4) -- 
(0,4)--cycle; 

\draw[thick] (0,1) -- (8,1); 
\draw[thick] (0,3) -- (8,3); 
\draw[thick] (0.5,0.5) -- (1.5,3.5); 
\draw[thick] (3.5,3.5)--(4.5,0.5);
\draw[thick] (3.5,0.5)--(5.5,3.5);
\draw[thick] (6.2,0.5)--(6.7,3.5);

\draw (2,3) node[below] {$C$};

\draw[rounded corners=0.3cm, red, thick] (2,1)--(2,1.5)node[right] {\small $\sigma_C(Q_0)$}
--(2.5,2)--(3,2.5)--(2.8,3);

\filldraw[fill=red, draw=red] (6.8,3) circle (2pt) ;
\filldraw[fill=red, draw=red] (7.2,1) circle (2pt) ;
\draw[thick, red] (6.8,3)--(7,2)--node[right]{\small $Q_0$} (7.2,1);


\filldraw[fill=red, draw=red] (1,2) node[left, red] {\small $p$} circle (2pt); 
\filldraw[fill=red, draw=red] (2.5,2) node[right, red] {\small $p''$} circle (2pt); 

\foreach \y in {1.05,1.5,2,2.5,3.05}
\draw[->, red, thick] (1,\y)--(1.7,\y);

\foreach \y in {1.05,1.5,2,2.5,3.05}
\draw[->, red, thick] (4,\y)--(3.3,\y);

\end{tikzpicture}
\caption{$V_1, p''=\sigma_C(\pi_C(p))$.}
\label{fig:degreeformula}
\end{figure}
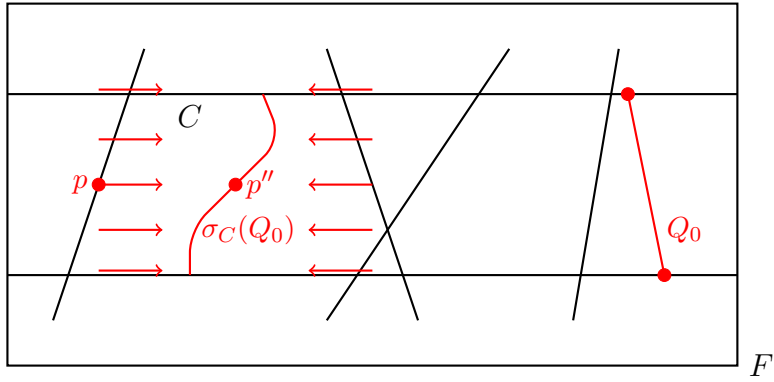

Then $tV_1+V_2$ is a vector field directed to $C$ (for $t\gg 0$). 
Since $C$ and $C^\lor$ are separated by $H\in\A\smallsetminus\wall_1(C)$, 
the vector field $-tV_1+V_2$ is directed to $C^\lor$. 
We can compute degree $\deg(C, C^\lor)$ using the vector field 
$-tV_1+V_2$. Note that $-tV_1(p)$ is outward vector field in along 
a $d$-dimensional space $G_1(p)$ and $V_2(p)$ is inward which is 
tangent to a $(\ell-1-d)$-dimensional space $G_2(p)$. 
Hence $\deg(C, C^\lor)$ is equal to the index of 
the following vector field in $\R^{\ell-1}$ at the origin. 
\begin{equation}
\label{eq:coord}
V=
\sum_{i=1}^d x_i\frac{\partial}{\partial x_i}-
\sum_{i=d+1}^{\ell-1} x_i\frac{\partial}{\partial x_i}, 
\end{equation}
where $d=\dim X(C)$. 
Recall that the de Rham cohomology group 
$H^{\ell-1}(S^{\ell-2})$ 
is generated by the differential form (\cite{BT}) 
\[
\sum_{i=1}^{\ell-1}(-1)^{i-1}x_idx_1\wedge\dots\wedge
\widehat{dx_i}\wedge\dots\wedge dx_{\ell-1}. 
\]
It is easily seen that the self map of 
$H^{\ell-1}(S^{\ell-2})$ 
induced by the Gauss map of the vector field (\ref{eq:coord}) is equal to 
the multiplication by $(-1)^{\ell-1-d}$. 
This completes the proof of Theorem \ref{thm:degree}.

\medskip

\noindent
{\bf Acknowledgements.} 
The present work was conducted during P. Bailet's stay at Hokkaido University 
as a postdoc. She thanks Postdoctoral Fellowship for Foreign Researchers 
(JSPS) for financial and other supports. 
M. Yoshinaga is partially supported by 
Grant-in-Aid for Scientific Research (C) 25400060 (JSPS).

\end{document}